\documentclass[12pt]{article}
\usepackage[latin9]{inputenc}
\usepackage[letterpaper]{geometry}
\usepackage{color}
\usepackage{amsmath}
\usepackage{amsthm}
\usepackage{amssymb}
\usepackage[unicode=true,pdfusetitle,
 bookmarks=true,bookmarksnumbered=false,bookmarksopen=false,
 breaklinks=false,pdfborder={0 0 1},backref=page,colorlinks=true]
 {hyperref}
\hypersetup{
 linkcolor=blue, citecolor=blue}

\makeatletter
\numberwithin{equation}{section}
\numberwithin{figure}{section}
\theoremstyle{plain}
\newtheorem{thm}{\protect\theoremname}[section]
\theoremstyle{definition}
\newtheorem{defn}[thm]{\protect\definitionname}
\theoremstyle{plain}
\newtheorem{cor}[thm]{\protect\corollaryname}
\theoremstyle{plain}
\newtheorem{prop}[thm]{\protect\propositionname}
\theoremstyle{plain}
\newtheorem{example}[thm]{\protect\examplename}

\@ifundefined{date}{}{\date{}}
\usepackage{graphicx}

\usepackage{enumitem}
\usepackage{tikz-cd}
\usepackage[backref=page]{hyperref}
\usepackage{cite}

\makeatother

\providecommand{\corollaryname}{Corollary}
\providecommand{\definitionname}{Definition}
\providecommand{\propositionname}{Proposition}
\providecommand{\theoremname}{Theorem}
\providecommand{\examplename}{Example}

\begin{document}
\global\long\def\F{\mathbf{F} }%
\global\long\def\Aut{\mathrm{Aut}}%
\global\long\def\C{\mathbb{C}}%
\global\long\def\H{\mathcal{H}}%
\global\long\def\U{\mathcal{U}}%
\global\long\def\P{\mathcal{P}}%
\global\long\def\ext{\mathrm{ext}}%
\global\long\def\hull{\mathrm{hull}}%
\global\long\def\triv{\mathrm{triv}}%
\global\long\def\Hom{\mathrm{Hom}}%

\global\long\def\trace{\mathrm{tr}}%
\global\long\def\End{\mathrm{End}}%

\global\long\def\L{\mathcal{L}}%
\global\long\def\W{\mathcal{W}}%
\global\long\def\E{\mathbb{E}}%
\global\long\def\SL{\mathrm{SL}}%
\global\long\def\R{\mathbb{R}}%
\global\long\def\Z{\mathbf{Z}}%
\global\long\def\rs{\to}%
\global\long\def\A{\mathcal{A}}%
\global\long\def\a{\mathbf{a}}%
\global\long\def\rsa{\rightsquigarrow}%
\global\long\def\D{\mathbf{D}}%
\global\long\def\b{\mathbf{b}}%
\global\long\def\df{\mathrm{def}}%
\global\long\def\eqdf{\stackrel{\df}{=}}%
\global\long\def\ZZ{\mathcal{Z}}%
\global\long\def\Tr{\mathrm{Tr}}%
\global\long\def\N{\mathbb{N}}%
\global\long\def\std{\mathrm{std}}%
\global\long\def\HS{\mathrm{H.S.}}%
\global\long\def\e{\varepsilon}%
\global\long\def\c{\mathbf{c}}%
\global\long\def\d{\mathbf{d}}%
\global\long\def\AA{\mathbf{A}}%
\global\long\def\BB{\mathbf{B}}%
\global\long\def\u{\mathbf{u}}%
\global\long\def\v{\mathbf{v}}%
\global\long\def\spec{\mathrm{spec}}%
\global\long\def\Ind{\mathrm{Ind}}%
\global\long\def\half{\frac{1}{2}}%
\global\long\def\Re{\mathrm{Re}}%
\global\long\def\Im{\mathrm{Im}}%
\global\long\def\p{\mathfrak{p}}%
\global\long\def\j{\mathbf{j}}%
\global\long\def\uB{\underline{B}}%
\global\long\def\tr{\mathrm{tr}}%
\global\long\def\rank{\mathrm{rank}}%
\global\long\def\K{\mathbf{K}}%
\global\long\def\hh{\mathcal{H}}%
\global\long\def\h{\mathfrak{h}}%

\global\long\def\EE{\mathcal{E}}%
\global\long\def\PSL{\mathrm{PSL}}%
\global\long\def\G{\mathcal{G}}%
\global\long\def\Int{\mathrm{Int}}%
\global\long\def\acc{\mathrm{acc}}%
\global\long\def\awl{\mathsf{awl}}%
\global\long\def\even{\mathrm{even}}%
\global\long\def\z{\mathbf{z}}%
\global\long\def\id{\mathrm{id}}%
\global\long\def\CC{\mathcal{C}}%
\global\long\def\cusp{\mathrm{cusp}}%
\global\long\def\new{\mathrm{new}}%

\global\long\def\LL{\mathbb{L}}%
\global\long\def\M{\mathbf{M}}%
\global\long\def\I{\mathcal{I}}%
\global\long\def\X{X}%
\global\long\def\free{\mathbf{F}}%
\global\long\def\into{\hookrightarrow}%
\global\long\def\Ext{\mathrm{Ext}}%
\global\long\def\B{\mathcal{B}}%
\global\long\def\Id{\mathrm{Id}}%
\global\long\def\Q{\mathbb{Q}}%

\global\long\def\O{\mathcal{T}}%
\global\long\def\Mat{\mathrm{Mat}}%
\global\long\def\NN{\mathrm{NN}}%
\global\long\def\nn{\mathfrak{nn}}%
\global\long\def\Tr{\mathrm{Tr}}%
\global\long\def\SGRM{\mathsf{SGRM}}%
\global\long\def\m{\mathbf{m}}%
\global\long\def\n{\mathbf{n}}%
\global\long\def\k{\mathbf{k}}%
\global\long\def\GRM{\mathsf{GRM}}%
\global\long\def\vac{\mathrm{vac}}%
\global\long\def\SS{\mathcal{S}}%
\global\long\def\red{\mathrm{red}}%
\global\long\def\V{V}%
\global\long\def\SO{\mathrm{SO}}%
\global\long\def\Gd{\Gamma^{\vee}}%
\global\long\def\fd{\mathrm{fd}}%
\global\long\def\perm{\mathrm{perm}}%
\global\long\def\tos{\xrightarrow{\mathrm{strong}}}%

\vspace{-3in} 
\title{$\SL_{4}(\Z)$ is not purely matricial field}
\author{Michael Magee and Mikael de la Salle}
\maketitle
\begin{abstract}
We prove that every finite dimensional unitary representation of $\SL_{4}(\Z)$
contains a non-zero $\SL_{2}(\Z)$-invariant vector. As a consequence,
there is no sequence of finite-dimensional representations of $\SL_{4}(\Z)$
that gives rise to an embedding of its reduced $C^{*}$-algebra into
an ultraproduct of matrix algebras.
\end{abstract}

\section{Statement of results}

We view $\SL_{2}(\Z)$ as the subgroup of $\SL_{4}(\Z)$ consisting
of matrices of the form $\left(\begin{array}{cccc}
* & * & 0 & 0\\
* & * & 0 & 0\\
0 & 0 & 1 & 0\\
0 & 0 & 0 & 1
\end{array}\right)$. The point of this note is to prove the following theorem.
\begin{thm}
\label{thm:main}Every finite dimensional unitary representation of
$\SL_{4}(\Z)$ contains a non-zero $\SL_{2}(\Z)$-invariant vector.
\end{thm}

We now explain some consequences of this theorem. 
\begin{defn}
If $\{\rho_{i}\}_{i=1}^{\infty}$ is a sequence of finite dimensional
unitary representations of a discrete group $\Gamma$, say $\{\rho_{i}\}_{i=1}^{\infty}$
\emph{strongly converges to the regular representation} if for any
$z\in\C[\Gamma]$,
\[
\lim_{i\to\infty}\|\rho_{i}(z)\|=\|\lambda_{\Gamma}(z)\|,
\]
where $\lambda_{\Gamma}:\Gamma\to U(\ell^{2}(\Gamma))$ is the left
regular representation. The norms above are operator norms. We write
$\rho_{i}\tos\lambda_{\Gamma}$ in this event.\footnote{Some authors include weak convergence --- that is, pointwise convergence
of normalized traces to the canonical tracial state on the reduced
group $C^{*}$-algebra --- in the definition of strong convergence.
In the case of $\SL_{4}(\Z)$, these definitions agree.}
\end{defn}

If $\Gamma$ is a discrete group, we say that $\Gamma$ is purely
matricial field if there is a sequence $\{\rho_{i}\}_{i=1}^{\infty}$
of finite dimensional unitary representations of $\Gamma$ such that
$\rho_{i}\tos\lambda_{\Gamma}$. In this case, if $\U$ is any free
ultrafilter on $\N$, not only does the sequence $\{\rho_{i}:\Gamma\to U(N_{i})\}_{i=1}^{\infty}$
induce an embedding 
\[
C_{r}^{*}(\Gamma)\xrightarrow{\varphi}\prod_{\U}\Mat_{N_{i}\times N_{i}}
\]
into the $C^{*}$-ultraproduct of matrix algebras, in which case $C_{r}^{*}(\Gamma)$
is \emph{matricial field} in the sense of Blackadar and Kirchberg
\cite{BK}, but also, there is a `lifting' of the embedding restricted
to the group algebra of the form

\[   \begin{tikzcd}     \C[\Gamma] \arrow{r}{} \arrow[swap]{dr}{\varphi} & \ell^{\infty}(\prod_{i\in\N}\Mat_{N_{i}\times N_{i}}) \arrow{d}{} \\      & \prod_{\U}\Mat_{N_{i}\times N_{i}}   \end{tikzcd} \]

See \cite[Appendix A]{BrownOzawa} for background on ultraproducts.
Here $\text{\ensuremath{\ell^{\infty}}(\ensuremath{\prod_{i\in\N}\Mat_{N_{i}\times N_{i}}})}$
is the collection of bounded sequences with respect to the $C^{*}$-norms.
See Schafhauser \cite{schafhauser2023finite} for a current overview
of MF reduced $C^{*}$-algebras of groups.
\begin{cor}
\label{cor:-is-not}$\SL_{4}(\Z)$ is not purely matricial field.
\end{cor}

This appears to be the first example of a finitely generated residually
finite group that is not purely matricial field. Groups that \emph{are}
known to be purely MF include free groups \cite{HaagerupThr}, limit
groups and surface groups \cite{louder2023strongly}, and right-angled
Artin groups, Coxeter groups, and hyperbolic three manifold groups
\cite{mageethomas}. 

It does not seem to be known whether $C^*_r(\SL_3(\Z))$ or $C^*_r(\SL_4(\Z))$ is MF in the sense of Blackadar and Kirchberg.

The property of a group being purely MF was historically relevant to the `$\Ext( C^*_r (F_2 ) )$ is not a group' problem (see \cite[Section 5.12]{voic_quasidiag}) and more  recently a strong form of purely MF for free groups, due to Bordenave and Collins \cite{BordenaveCollins}, was used to prove Buser's conjecture on the bottom of the spectrum of hyperbolic surfaces in two  different ways \cite{HideMagee, louder2023strongly}.
\begin{proof}[Proof of Corollary \ref{cor:-is-not}]
Let $S$ and $T$ denote standard generators of $\SL_{2}(\Z)$. Theorem
\ref{thm:main} implies that for any finite dimensional representation
$\rho$ of $\SL_{4}(\Z)$,
\[
\|\rho(S+S^{-1}+T+T^{-1})\|=4.
\]
On the other hand, as an $\SL_{2}(\Z)$-module, $\ell^{2}(\SL_{4}(\Z))$
breaks up into a direct sum of copies of $\ell^{2}(\SL_{2}(\Z))$.
Since $\SL_{2}(\Z)$ is not amenable, we have 
\[
\|\lambda_{\SL_{4}(\Z)}(S+S^{-1}+T+T^{-1})\|=\|\lambda_{\SL_{2}(\Z)}(S+S^{-1}+T+T^{-1})\|<4.\qedhere
\]
\end{proof}
Theorem \ref{thm:main} does not hold with `four' replaced by `three',
since for primes $p$ there are nontrivial irreducible representations of
$\SL_{3}(\Z/p\Z)$ without non-zero $\SL_{2}(\Z/p\Z)$-invariant vectors
(P. Deligne, private communication, see Example~\ref{ex:deligne}). Nevertheless it could still be
the case that $\SL_{3}(\Z)$ is not purely MF and we would be very
interested to know the answer of this question. It would perhaps clarify
the relation between property (T) and purely MF --- as far as we
know there is no direct relation. Property (T) says that it is difficult
to approach finite dimensional representations by arbitrary ones whereas
the group not being purely matricial field says that it is difficult
to approach the regular representation by finite-dimensional ones. 

\subsection*{Acknowledgments}

We thank Pierre Deligne for explaining to us the above mentioned fact
about representations of $\SL_{3}(\Z/p\Z)$. We thank Kevin Boucher, Yves de Cornulier and Olivier Dudas for comments and conversations about this project.

Funding: 

M. M. This material is based upon work supported by the National Science
Foundation under Grant No. DMS-1926686. This project has received
funding from the European Research Council (ERC) under the European
Union\textquoteright s Horizon 2020 research and innovation programme
(grant agreement No 949143).

M. S. Research supported by the Charles Simonyi Endowment at the Institute
for Advanced Study, and the ANR project ANCG Project-ANR-19-CE40-0002.

\section{Proofs of results}

It is an elementary consequence of work of Bass-Milnor-Serre on the
congruence subgroup property \cite{BassMilnorSerre} (e.g. \cite[\S 5]{Bekka}) that every finite dimensional unitary representation of $\SL_{4}(\Z)$ arises from a composition of homomorphisms
\[
\SL_{4}(\Z)\to\SL_{4}(\Z/N\Z)\xrightarrow{\phi}U(M)
\]
for some $N\in\N$. To prove Theorem \ref{thm:main} it therefore
suffices to prove the following.
\begin{prop}
\label{prop:levelN}For all $N\in\N$, every non-trivial finite dimensional representation
$\phi$ of $\SL_{4}(\Z/N\Z)$ has a non-zero $\SL_{2}(\Z/N\Z)$-invariant
vector.
\end{prop}

As before $\SL_{2}(\Z/N\Z)$ is the collection of matrices of the
form $\left(\begin{array}{cccc}
* & * & 0 & 0\\
* & * & 0 & 0\\
0 & 0 & 1 & 0\\
0 & 0 & 0 & 1
\end{array}\right)$ in $\SL_{4}(\Z/N\Z)$. The rest of the paper proves Proposition \ref{prop:levelN}.
We may assume that $\phi$ is irreducible and moreover that it is
\emph{new}, meaning that it does not factor through reduction modulo
$N'$
\[
\SL_{4}(\Z/N\Z)\to\SL_{4}(\Z/N'\Z)
\]
for any $N'<N$ dividing $N$. (Or else we replace $N$ by $N'$.)

\subsection{Reduction to prime powers}

Let 
\[
N=\prod_{\text{\ensuremath{p} prime}}p^{e(p)}
\]
be the prime factorization of $N$. By the Chinese remainder theorem
\[
\SL_{4}(\Z/N\Z)\cong\prod_{\text{\ensuremath{p} prime,}e(p)>0}\SL_{4}(\Z/p^{e(p)}\Z)
\]
and this induces a splitting 
\[
\phi\cong\bigotimes_{\text{\ensuremath{p} prime,}e(p)>0}\phi_{p}
\]
 where $\phi_{p}$ are irreducible representations of $\SL_{4}(\Z/p^{e(p)}\Z)$.
The assumption that $\phi$ is new implies that each $\phi_{p}$ is
new. If we can prove all the $\phi_{p}$ have non-zero $\SL_{2}(\Z/p^{e(p)}\Z)$-invariant
vectors $v_{p}$, then
\[
v=\bigotimes_{\text{\ensuremath{p} prime,}e(p)>0}v_{p}
\]
 will be the required non-zero invariant vector for $\SL_{2}(\Z/N\Z)\cong\prod_{\text{\ensuremath{p} prime,}e(p)>0}\SL_{2}(\Z/p^{e(p)}\Z)$
--- the inclusion of $\SL_{2}$ in $\SL_{4}$ that we use commutes
with our applications of the Chinese remainder theorem.

The strategy of the proof is the following:
\begin{description}
\item [{Step~1:}] We prove the representation is non-trivial when restricted
to all elementary cyclic subgroups of level $p^{r-1}$.
\item [{Step~2:}] We use Step 1 to prove that on restriction to a particular
copy of the Heisenberg group modulo $p^{r}$, we find a particular
type of character, namely, the one described in (\ref{eq:v-action}).
\item [{Step~3:}] We take a non-zero vector in the isotypic subspace of
the character of the Heisenberg group found in Step 2. By averaging
this vector over a copy of $\SL_{2}(\Z/p^{r}\Z$) we find a non-zero
$\SL_{2}(\Z/p^{r}\Z)$-invariant vector. Here, the form of the Heisenberg
group character we found in the previous step is important to make
sure that this average is non-zero.
\end{description}

\subsection{Prime powers: step 1}

It therefore now suffices to prove Proposition \ref{prop:levelN}
when $N=p^{r}$, $r\geq1$. Let $\phi$ denote the irreducible representation.
For $1\leq i\neq j\leq4$ let $\e_{ij}$ denote the matrix with one
in the $i,j$ entry and zeros elsewhere. The first step is to find
a non-trivial subrepresentation of some 
\[
C_{ij}\eqdf\langle I+p^{r-1}\e_{ij}\rangle.
\]
As $C_{ij}$ is abelian, by further passing to a subrepresentation,
we may assume the non-trivial subrepresentation is irreducible and
hence a character.

If $r=1$ $\SL_{4}(\Z/p\Z)$ is generated by such cyclic subgroups.
So suppose for this step that $r>1$. 

We could proceed by using a result of Bass--Milnor--Serre \cite[Cor. 4.3.b]{BassMilnorSerre}
--- stating that the principal congruence subgroup of level $p^{r}$
in $\SL_{4}(\Z)$ is normally generated by elementary matrices. For
completeness, below we give a simple self-contained proof of what
we need.

Let $G(p^{r-1})$ denote the kernel of reduction mod $p^{r-1}$ on
$\SL_{4}(\Z/p^{r}\Z$). Since we assume $\phi$ is new, we know $G(p^{r-1})$
is not contained in the kernel of $\phi$. Let $\Mat_{4\times4}^{0}(\Z/p\Z)$
denote the four by four matrices with entries in $\Z/p\Z$ and zero
trace. The map 
\begin{equation}
A\in\Mat_{4\times4}^{0}(\Z/p\Z)\mapsto I+p^{r-1}A\in G(p^{r-1})\label{eq:add-isomorphism}
\end{equation}
 is easily seen to be an isomorphism of groups, where the group law
on $\Mat_{4\times4}^{0}(\Z/p\Z)$ is addition. 

We want to first show that some $C_{ij}$ acts non-trivially in the
representation.

Suppose for a contradiction that we do not find a non-trivial irreducible
subrepresentation of some $C_{ij}$, so that all $1+p^{r-1}B$ with
$B$ zero on the diagonal are in $\ker(\phi)$. Using (\ref{eq:add-isomorphism}),
this assumption implies that $\phi$ restricted to $G(p^{r-1})$ is
equivalent to a non-trivial representation of 
\[
\Mat_{4\times4}^{0}(\Z/p\Z)/\{\text{elements of \ensuremath{\Mat_{4\times4}^{0}(\Z/p\Z)} that are zero on the diagonal\}.}
\]
But this is spanned by equivalence classes of diagonal elements. Thus there is necessarily a diagonal matrix $A$ such that $I+p^{r-1}A$ is
not in the kernel of $\phi$, without loss of generality (choosing
a basis for the diagonal trace zero matrices) $A=\left(\begin{array}{cccc}
1 & 0 & 0 & 0\\
0 & -1 & 0 & 0\\
0 & 0 & 0 & 0\\
0 & 0 & 0 & 0
\end{array}\right).$

We calculate 
\begin{align*}
 & \left(\begin{array}{cccc}
1 & 0 & 0 & 0\\
1 & 1 & 0 & 0\\
0 & 0 & 1 & 0\\
0 & 0 & 0 & 1
\end{array}\right)\left(I+p^{r-1}\left(\begin{array}{cccc}
0 & 1 & 0 & 0\\
0 & 0 & 0 & 0\\
0 & 0 & 0 & 0\\
0 & 0 & 0 & 0
\end{array}\right)\right)\left(\begin{array}{cccc}
1 & 0 & 0 & 0\\
-1 & 1 & 0 & 0\\
0 & 0 & 1 & 0\\
0 & 0 & 0 & 1
\end{array}\right)\\
 & =I+p^{r-1}\left(\begin{array}{cccc}
-1 & 1 & 0 & 0\\
-1 & 1 & 0 & 0\\
0 & 0 & 0 & 0\\
0 & 0 & 0 & 0
\end{array}\right)\in\ker(\phi).
\end{align*}
Then also 
\begin{align*}
 & \left(I+p^{r-1}\left(\begin{array}{cccc}
-1 & 1 & 0 & 0\\
-1 & 1 & 0 & 0\\
0 & 0 & 0 & 0\\
0 & 0 & 0 & 0
\end{array}\right)\right)\left(I+p^{r-1}\left(\begin{array}{cccc}
0 & -1 & 0 & 0\\
1 & 0 & 0 & 0\\
0 & 0 & 0 & 0\\
0 & 0 & 0 & 0
\end{array}\right)\right)\\
 & =I+p^{r-1}\left(\begin{array}{cccc}
-1 & 0 & 0 & 0\\
0 & 1 & 0 & 0\\
0 & 0 & 0 & 0\\
0 & 0 & 0 & 0
\end{array}\right)\in\ker(\phi),
\end{align*}
a contradiction. The conclusion of this step is no matter $r\geq1$,
we find $i\neq j$ such that $C_{ij}\notin\ker(\phi)$. But in fact,
since all $C_{ij}$ are conjugate in $\SL_{4}(\Z/p\Z)$, this means
that:

\emph{No $C_{ij}$ is contained in the kernel of $\phi$.}

\subsection{Prime powers: step 2 }

Let $U_{1}$ denote the group 
\[
U_{1}\eqdf\left\{ \Upsilon(u_{1},u_{2},u_{3})\eqdf\left(\begin{array}{cccc}
1 & 0 & 0 & u_{1}\\
0 & 1 & 0 & u_{2}\\
0 & 0 & 1 & u_{3}\\
0 & 0 & 0 & 1
\end{array}\right)\right\} \leq\SL_{4}(\Z/p^{r}\Z).
\]
The group $U_{1}$ is isomorphic to $(\Z/p^{r}\Z,+)^{3}$ so the
restriction of $\phi$ to $U_{1}$ breaks into a direct sum of
one-dimensional subspaces where $U_{1}$ acts by a character. Moreover, $\SL_{3}(\Z/p^r \Z)$ normalizes $U_{1}$ so it acts on the characters of $U_1$ appearing like this by $g \chi = \chi( g^{-1} \cdot g)$. This action is called the dual action. Every such character is of the form
\begin{equation}
\chi\colon\Upsilon(u_{1},u_{2},u_{3})\mapsto\exp\left(2\pi i\frac{(\xi_{1}u_{1}+\xi_{2}u_{2}+\xi_{3}u_{3})}{p^{r}}\right)\label{eq:char-form}
\end{equation}
for $(\xi_{1},\xi_{2},\xi_{3}) \in (\Z/p^r\Z)^3$ and the dual action corresponds to $(\xi_{1},\xi_{2},\xi_{3})\mapsto (\xi_{1},\xi_{2},\xi_{3}) g^{-1}$. If $(\xi_{1},\xi_{2},\xi_{3})\equiv0\bmod p$ then all $\Upsilon(u_{1},u_{2},u_{3})$
with $p^{r-1}|u_{1},u_{2},u_{3}$ are in the kernel of the character.
If all obtained characters satisfy this condition, then $\phi$ restricted
to $U_{1}$ has $U_{1}\cap G(p^{r-1})$ in its kernel. But by Step
1, $C_{14}$ is not contained in the kernel of $\phi$. Hence in the
restriction of $\phi$ to $U_{1}$ there must be a character of the
form (\ref{eq:char-form}) where $(\xi_{1},\xi_{2},\xi_{3})\not\equiv0\bmod p$.
Since $\SL_{3}(\Z/p^r\Z)$ acts transitively on the vectors in $(\Z/p^r\Z)^{3}$ satisfying  $(\xi_{1},\xi_{2},\xi_{3})\not\equiv0\bmod p$, by considering the dual action we may assume
\[
(\xi_{1},\xi_{2},\xi_{3})=(0,0,1).
\]
Let $V_{\chi}$ be the $\chi$-isotypic space for the restriction
of $\phi$ to $U_{1}$, where $\chi$ and $\xi$ are as above.

The group
\[
G_{1}\eqdf\left\{ \left(\begin{array}{cccc}
* & * & * & 0\\
* & * & * & 0\\
0 & 0 & 1 & 0\\
0 & 0 & 0 & 1
\end{array}\right)\right\} \leq\SL_{4}(\Z/p^{r}\Z)
\]
normalizes $U_{1}$ and fixes $\chi$ under the dual action. Hence
$V_{\chi}$ is an invariant subspace for $G_{1}$. Now restrict $V_{\chi}$
to the group 
\[
U_{2}\eqdf\left\{ [v_{1};v_{2}]\eqdf\left(\begin{array}{cccc}
1 & 0 & v_{1} & 0\\
0 & 1 & v_{2} & 0\\
0 & 0 & 1 & 0\\
0 & 0 & 0 & 1
\end{array}\right)\right\} \leq G_{1}
\]
and we will decompose this into characters $\theta$ of $U_{2}$;
let $V_{\chi,\theta}$ denote the corresponding isotypic subspace. 

Consider now the group

\[
H\eqdf\left\{ [x;y;z]\eqdf\left(\begin{array}{cccc}
1 & 0 & 0 & 0\\
0 & 1 & x & z\\
0 & 0 & 1 & y\\
0 & 0 & 0 & 1
\end{array}\right)\right\} \leq\SL_{4}(\Z/p^{r}\Z).
\]
As we already mentioned, $G_{1}$ preserves $V_{\chi}$. Obviously
$U_{1}$ fixes all its characters under the dual action induced by
conjugation, hence all $\left(\begin{array}{cccc}
* & * & * & *\\
* & * & * & *\\
0 & 0 & 1 & *\\
0 & 0 & 0 & 1
\end{array}\right)$ fix our chosen $\chi$ under the dual action, or in other words,
leave $V_{\chi}$ invariant. Hence the space $V_{\chi}$ is invariant
by $H$. 

We have $[0;0;z]=\Upsilon(0,z,0)\in U_{1}$ and for $v\in V_{\chi}$
\[
\Upsilon(0,z,0)v=\exp\left(2\pi i\frac{(0\cdot0+0\cdot z+1\cdot0)}{p^{r}}\right)v=v.
\]
Hence the action of $H$ on $V_{\chi}$ has kernel that contains the
subgroup with $x=y=0$, which is isomorphic to $\Z/p^{r}\Z$. Hence
the action of $H$ on $V_{\chi}$ factors through an action of 
\[
H/(\Z/p^{r}\Z)\cong(\Z/p^{r}\Z)^{2}.
\]
We want to find a particular character of $H$ and to do so we split
into the following cases.

\textbf{Case 1. $V_{\chi}$ restricted to $U_{2}$ is trivial.} Then
obviously $H$ acts on all of $V_{\chi}$ by 
\begin{equation}
[x;y;z]\mapsto\exp\left(2\pi i\frac{y}{p^{r}}\right).\label{eq:desired-character}
\end{equation}

\textbf{Case 2. }Otherwise, we find a character $\theta$ in $V_{\chi}$
of the form
\[
\theta:[v_{1};v_{2}]\mapsto\exp\left(2\pi i\frac{(\zeta_{1}v_{1}+\zeta_{2}v_{2})}{p^{r}}\right)
\]
with $(\zeta_{1},\zeta_{2})\not\equiv(0,0)\bmod p^{r}$. Write $(\zeta_{1},\zeta_{2})=p^{R}(z_{1},z_{2})$
with $(z_{1},z_{2})\neq(0,0)\bmod p$. By conjugation in $\SL_{2}(\Z/p^{r}\Z)\leq G_{1}$
--- which normalizes $U_{2}$ --- we can find a new $\theta'$ with
corresponding $z_{1}=1,z_{2}=0$ so that
\[
\theta':[v_{1};v_{2}]\mapsto\exp\left(2\pi i\frac{v_{1}}{p^{r-R}}\right).
\]
In particular, on $V_{\chi,\theta'}$ $H$ acts by the character (\ref{eq:desired-character}). 

\emph{To summarize, in any case, there exists a non-zero vector $v\in V_{\chi}$
such that 
\begin{equation}
\phi([x;y;z])v=\exp\left(2\pi i\frac{y}{p^{r}}\right)v.\label{eq:v-action}
\end{equation}
}

\subsection{Prime powers: step 3}

Now let 
\[
G_{2}\eqdf\left\{ \left(\begin{array}{cccc}
1 & 0 & 0 & 0\\
0 & a & b & 0\\
0 & c & d & 0\\
0 & 0 & 0 & 1
\end{array}\right)\right\} \leq\SL_{4}(\Z/p^{r}\Z).
\]
From (\ref{eq:v-action}), $v$ is fixed by the subgroup
\[
N\eqdf\left\{ \left(\begin{array}{cccc}
1 & 0 & 0 & 0\\
0 & 1 & n & 0\\
0 & 0 & 1 & 0\\
0 & 0 & 0 & 1
\end{array}\right)\right\} \leq G_{2}.
\]
This implies that if $W$ denotes the representation of $G_{2}\cong\SL_{2}(\Z/p^{r}\Z)$
generated by $v$, that $W$ is a quotient of the induced representation
\[
\Ind_{N}^{G_{2}}\triv=\C[G_{2}]\otimes_{N}\C.
\]

Suppose $g\in G_{2}$, with $a,b,c,d$ as above in $\Z/p^{r}\Z$.
We have 
\begin{align*}
\phi([0;y;z])\phi(g^{-1})v & =\phi(g^{-1})\phi(g[0;y;z]g^{-1})v\\
 & =\phi(g^{-1})\phi([0;cz+dy;az+by])v\\
 & =\phi(g^{-1})\exp\text{\ensuremath{\left(2\pi i\frac{(dy+cz)}{p^{r}}\right)}}v.
\end{align*}
This means, in this co-adjoint action of $G_2$ on characters of the group $\langle[0;y;z]\rangle$,
$N$ is precisely the stabilizer of the character of $v$, and hence
\[
\dim W=|G_{2}|/|N|=\dim\Ind_{N}^{G_{2}}\triv,
\]
so in fact, $W\cong\Ind_{N}^{G_{2}}\triv$ as a $G_{2}$ representation.
By Frobenius reciprocity, this contains the trivial representation
of $G_{2}$. Finally, $G_{2}$ and the upper left copy of $\SL_{2}(\Z/p^{r}\Z)$
are conjugate in $\SL_{4}(\Z/p^{r}\Z)$. This concludes the proof. 
\subsection{Representations of $\SL_3(\Z/p\Z)$}
The character tables of $\SL_3(\F)$ for finite fields $\F$ have been computed in \cite{SimpsonSutherland}. In particular, if we view $\SL_{2}(\F)$ as the subgroup of $\SL_{3}(\F)$ consisting
of matrices of the form $\left(\begin{array}{ccc}
* & * & 0 \\
* & * & 0 \\
0 & 0 & 1 \\
                               \end{array}\right)$, we obtain the following example, explained to us by Deligne:
\begin{example}\label{ex:deligne} For every prime power $q$, $\SL_3(\F_q)$ has an irreducible representation such that, for every $g \in \SL_2(\F_q)$,
  \[ \Tr(\pi(g)) = \begin{cases} (q-1)(q^2-1)& \textrm{if }g=1\\1-q &\textrm{if }(g-1)^2 = 0 \neq g-1\\0& \textrm{if }(g-1)^2 \neq 0.\end{cases}\]
This representation does not have a non-zero $\SL_2(\F_q)$-invariant vector.
\end{example}
The representations are any of those denoted $\chi_{r^2s}(u)$ in \cite[Table 1b]{SimpsonSutherland} (that are associated with tori of split rank $0$ in the Deligne-Lusztig theory \cite{Humphreys}). The properties of $\Tr(\pi(g))$ follow readily from this table and the description of the conjugacy classes in $\SL_2(\F_q)$ (e.g. \cite[Section 1.3]{bonnafe}).

Such a representation does not have non-zero $\SL_2(\F_q)$-invariant vectors because, using that there are exactly $q^2-1$ unipotent matrices in $\SL_2(\F_q)\setminus\{1\}$ \cite[Section 1.3]{bonnafe}, we can compute that the trace of the projection on the $\SL_2(\F_q)$-invariant vectors is $0$:
  \[ \Tr\Big(\sum_{g \in \SL_2(\F_q)} \pi(g)\Big) = 1 \cdot (q-1)(q^2-1) + (q^2-1)\cdot(1-q) = 0.\]

\noindent \bibliographystyle{amsalpha}
\bibliography{sl4}

\noindent Michael Magee, \\
Department of Mathematical Sciences, Durham University, Lower Mountjoy,
DH1 3LE Durham, UK\\
IAS Princeton, School of Mathematics, 1 Einstein Drive, Princeton
08540, USA\\
\texttt{michael.r.magee@durham.ac.uk}~\\

\noindent Mikael de la Salle, \\
Institut Camille Jordan, CNRS, Universit\'{e} Lyon 1, France\\
IAS Princeton, School of Mathematics, 1 Einstein Drive, Princeton
08540, USA\\
\texttt{delasalle@math.univ-lyon1.fr}\\

\end{document}